\documentclass[12pt,amstex]{amsart}
\usepackage{stmaryrd}
\usepackage{epsfig}
\usepackage{enumerate}
\usepackage{amsmath}
\usepackage{amssymb}
\usepackage{amscd}
\usepackage{graphicx}
\usepackage{lineno}
\usepackage{pstricks}

\usepackage{tocvsec2}

\allowdisplaybreaks[4]

\usepackage [backref, colorlinks, citecolor=red, anchorcolor=black, linkcolor=blue]{hyperref}

\usepackage{amsfonts}
\usepackage[top=35mm, bottom=35mm, left=30mm, right=30mm]{geometry}
\usepackage{verbatim}
\usepackage{graphicx, amssymb, amsmath,amsthm}
\allowdisplaybreaks[4]
\usepackage{appendix}
 \usepackage{mathrsfs}

\usepackage{xcolor}

\newcommand{\mb}{\mathbb}

   \def \r{\rho}

\theoremstyle{definition}
\newtheorem{theorem}{Theorem}[section]
\newtheorem{lemma}[theorem]{Lemma}

\newtheorem{definition}[theorem]{Definition}
\newtheorem{remark}[theorem]{Remark}

\newtheorem{proposition}[theorem]{Proposition}

\newtheorem{claim}[theorem]{Claim}

\numberwithin{equation}{section}

\renewcommand*{\backref}[1]{}
\renewcommand*{\backrefalt}[4]{\quad \tiny
  \ifcase #1 (\textbf{NOT CITED.})%
  \or    (Cited on Section~#2.)%
  \else   (Cited on Section~#2.)%
  \fi}

\makeatletter

\def\MRbibitem{\@ifnextchar[\my@lbibitem\my@bibitem}

\def\mybiblabel#1#2{\@biblabel{{\hyperref{http://www.ams.org/mathscinet-getitem?mr=#1}{}{}{#2}}}}

\def\myhyperanchor#1{\Hy@raisedlink{\hyper@anchorstart{cite.#1}\hyper@anchorend}}

\def\my@lbibitem[#1]#2#3#4\par{%
  \item[\mybiblabel{#2}{#1}\myhyperanchor{#3}\hfill]#4%
  \@ifundefined{ifbackrefparscan}{}{\BR@backref{#3}}%
  \if@filesw{\let\protect\noexpand\immediate
    \write\@auxout{\string\bibcite{#3}{#1}}}\fi\ignorespaces%
}

\def\my@bibitem#1#2#3\par{%
  \refstepcounter\@listctr
  \item[\mybiblabel{#1}{\the\value\@listctr}\myhyperanchor{#2}\hfill]#3%
  \@ifundefined{ifbackrefparscan}{}{\BR@backref{#2}}%
  \if@filesw\immediate\write\@auxout
    {\string\bibcite{#2}{\the\value\@listctr}}\fi\ignorespaces%
}

\makeatother

\begin{document}

\title[]{Invariant tori for a class of affined Anosov mappings with quasi-periodic forces}

\date{}

\author{Xinyu Bai} \address[Xinyu Bai] {School of Mathematics\\
Sichuan University \\
  Chengdu, Sichuan, 610016, China}
\email[X. Bai]{23210180080@m.fudan.edu.cn}

\author{Zeng Lian} \address[Zeng Lian] {School of Mathematics\\
Sichuan University \\
  Chengdu, Sichuan, 610016, China}
\email[Z. Lian]{lianzeng@scu.edu.cn, zenglian@gmail.com}

\author{Xiao Ma} \address[Xiao Ma] {CAS Wu Wen-Tsun Key Laboratory of Mathematics, School of Mathematical Sciences \\
University of Science and Technology of China \\
  Hefei, Anhui, China}
\email[X. Ma]{xiaoma@ustc.edu.cn}

\author{Hang Zhao} \address[Hang Zhao] {School of Mathematics\\
Sichuan University \\
  Chengdu, Sichuan, 610016, China}
\email[H. Zhao]{hangzhaoscu@163.com}



\begin{abstract} In this paper, we consider a class of  affined Anosov mappings with quasi-periodic forces, and show that there is a unique positive integer $m$, which only depends on the system, such that the exponential growth rate of the cardinality of invariant tori of degree $m$ is equal to the topological entropy.
\end{abstract}
\maketitle


\section{Introduction}
\subsection{Background and motivation}
To describe the complexity of  dynamical systems, which is one of the most important subjects in the study of dynamical systems,  mathematicians introduced lots of indices, such as Lyapunov exponents, entropy, Smale horseshoe et al.
In 1965, Adler, Konheim, and McAndrew\cite{AKM65} introduced the concept of topological entropy using open covers. The widely accepted definition of topological entropy for systems on metric spaces, provided by Bowen\cite{Bow68, Bow711}, uses separating sets and spanning sets to describe the exponential growth rate of the number of separated orbit segments. 
The distribution properties of periodic orbits play an essential role in studying the complexity of dynamical systems. 
For Axiom A maps and flows, Bowen\cite{Bow712, Bow72} proved the denseness of periodic orbits  and  its exponential growth rate of the cardinality is equal to the topological entropy. 
For non-uniformly hyperbolic systems, Katok\cite{Kat80} proved that the exponential growth rate of periodic points can approximate any given hyperbolic measures in terms of metric entropy. 

Non-autonomous and random dynamical systems often lack periodic orbits due to the absence of recurrence, which creates one of the most significant difficulties in establishing the conclusions of classical dynamical systems in these systems. 
In 2001, Klünger\cite{Kl01} introduced the concept of random periodic points. 
Zhao and Zheng\cite{ZZ09} proved the existence of random periodic solutions for a class of stochastic differential equations. 
Recently, Huang, Lian, and Lu\cite{HLL} proved the existence and the density of random periodic orbits in Anosov systems driven by quasi-periodic forces. 
In particular, they pointed out that when the period is sufficiently large, the cardinality of random periodic orbits will become uncountable. 
Therefore, its exponential growth rate cannot be related to the system's topological entropy. 
It is worth noting that these random periodic points are Borel measurable mappings from the sample space (base space) to the phase space (fiber). 
In \cite{HLL},  they also showed the failure of the mechanism of generating invariant tori through continuous random periodic point(Lemma 7.2 in \cite{HLL}). 
However, for an affined Arnold's cat map driven by a quasi-periodic force, they gave some conditions to assure the existence of an invariant torus, which sparked our interest in discovering the relationship between the exponential growth rate of the cardinality of invariant tori and the topological entropy.

\subsection{Setting and results}
In this paper, we consider the skew-product system driven by an irrational rotation on $\mathbb{T}:=\mathbb{R}/ \mathbb{Z}$.

Let $R_{\alpha}: \mathbb{T}\to \mathbb{T}, \omega\mapsto \omega+\alpha \mod 1$, $\alpha\in \mathbb{R}\setminus \mathbb{Q}$ be the irrational rotation on $\mathbb{T}$, $C(\mathbb{T},\mathbb{T}^2)$ be the space of continuous mappings from $\mathbb{T}$ to $\mathbb{T}^2$.
We consider the system $(\mathbb{T}\times\mathbb{T}^2, \varphi)$ as follows. 
\begin{equation}\label{system0}
\begin{aligned}
\varphi: &\mathbb{T}\times \mathbb{T}^2 \to  \mathbb{T}\times \mathbb{T}^2\\
&(\omega, x)\mapsto (R_{\alpha}(\omega),f_{\omega}(x)),
\end{aligned}
\end{equation}
where $f: \mathbb{T}\to C^{1}(\mathbb{T}^2, \mathbb{T}^2)$ and $f_{\omega}:=f(\omega)$ for convenience.

In this paper, we consider  the following type of invariant structures of $(\mathbb{T}\times \mathbb{T}^2,\varphi)$.
\begin{definition}[$\varphi^n$-invariant torus]\label{def1}
For the system $(\mathbb{T}\times \mathbb{T}^2,\varphi)$ as in \eqref{system0}, if there are $m, n\in \mathbb{N}$ and a continuous map $g: \mathbb{R}\to\mathbb{T}^2$, such that for any $\omega\in \mathbb{R}$, one has
\begin{itemize}
\item[(1)] $g(\omega)=g(\omega + km)$ for any $k\in \mathbb{Z}$;
\item[(2)] $ \varphi^n(\omega \mod 1, g(\omega) )= (\omega+n\alpha \mod 1, g(\omega+n\alpha ));$
\item[(3)] $\{g(\omega),\dots, g(\omega+m-1)\}$ are pairwise distinct.
\end{itemize}
The graph of a multi-valued map $$g_{\mathcal{T}}: \mathbb{T}\to \mathbb{T}^2, \omega\mapsto \left\{g(\omega+i): i\in \{0,\dots, m-1\}\right\}$$  is called a {\bf $\varphi^n$-invariant torus of degree $m$}, and denoted by $$\mathcal{T}:=\{(\omega,g_{\mathcal{T}}(\omega)): \forall \omega\in \mathbb{T}\}.$$
\end{definition}
Throughout this paper, the constant $m\in \mathbb{N}$ is called the degree of $\mathcal{T}$, and denoted by $deg({\mathcal{T}})$. 
We denote by $\mathcal{G}(\varphi)$ the collection of invariant tori of $(\mathbb{T}\times \mathbb{T}^2, \varphi)$, and $\mathcal{G}(\varphi; n,m)$ the collection of all $\varphi^k(k\leq n)$-invariant tori of degree $m$. 

In this paper, we consider affined Anosov mappings with quasi-periodic forces:
\begin{equation}\label{syst1}
\begin{aligned}
\varphi: &\mathbb{T}\times \mathbb{T}^2 \to  \mathbb{T}\times \mathbb{T}^2\\
&(\omega, x)\mapsto (R_{\alpha}(\omega), Ax+ h(\omega)),
\end{aligned}
\end{equation}
where
$h(\omega)= \begin{pmatrix}
h_1(\omega)\\
h_2(\omega)\end{pmatrix} \in C(\mathbb{T}, \mathbb{T}^2)$, the matrix $A\in GL(2, \mathbb{Z})$ is hyperbolic, and $GL(2, \mathbb{Z})$ is the general linear group of order 2 over  $\mathbb{Z}$. 

Note that there are two kinds of degrees in this paper, the first is the degree of the invariant torus we defined in Definition \ref{def1}; and the second is the degree of a continuous mapping $f:\mathbb{T}\to \mathbb{T}$, which is denoted by $deg(f)$.
Additionally, $I-A$ is invertible as $A\in GL(2, \mathbb{Z})$ is hyperbolic. 
Throughout this paper, we use $\sharp$ to represent the cardinality of the corresponding set.
\begin{theorem}\label{thm1}
For $\varphi: \mathbb{T}\times \mathbb{T}^2 \to \mathbb{T}\times \mathbb{T}^2 $ as in \eqref{syst1}, let $m\in \mathbb{N}$ be the smallest positive integer satisfying 
$$m\cdot (I-A)^{-1}\begin{pmatrix} deg(h_1)\\ deg(h_2)\end{pmatrix}\in \mathbb{Z}^2.$$   
Then, one has
\begin{itemize}
\item[(1)] $\mathcal{G}(\varphi)=\cup_{n\in \mathbb{N}}\mathcal{G}(\varphi;n,m)$;
\item[(2)] $\sharp \mathcal{G}(\varphi;n,m)<+\infty$ for all $n\in \mathbb{N}$;
\item[(3)] $ \lim_{n\rightarrow +\infty}\frac1n \log \sharp\mathcal{G}(\varphi; n,m)=h_{top}(\varphi).$
\end{itemize}
\end{theorem}

To get the above theorem, we firstly show the existence of a $\varphi$-invariant torus in Sect. \ref{sect2}. Then, we prove Theorem \ref{thm1} in Sect. \ref{sect4} through some properties of the homogeneous system  in Sect. \ref{sect3}.

\section{Existence of a $\varphi$-invariant torus}\label{sect2}
In this section, we show that there is a $\varphi$-invariant torus of a particular degree.
\begin{proposition}\label{prop2.1.1}
For system $( \mathbb{T}\times \mathbb{T}^2, {\varphi})$ as in \eqref{syst1}, let $m$ be the smallest positive integer such that
\begin{equation}
\label{e3}
m\cdot (I-A)^{-1}\left(\begin{matrix}deg({h}_1)\\deg({h}_2)\end{matrix} \right)\in\mathbb{Z}^2.
\end{equation} 
Then, there exists a ${\varphi}$-invariant torus $\mathcal{T}$ of $deg(\mathcal{T})=m$.
\end{proposition}
At first, we give a relationship between random periodic points  and  invariant tori.  
Recall the definition of random periodic points as in \cite{HLL}.
For a given system $(\mathbb{T}\times\mathbb{T}^2,\varphi)$ as in \eqref{system0}, the set $\{graph(g_i)| g_i\in L^{\infty}(\mathbb{T},\mathbb{T}^2)\}_{1\leq i\leq n}$ is called a {\bf random periodic orbit} of $(\mathbb{T}\times \mathbb{T}^2,\varphi)$ of period  $n$ if 
 $$ \varphi(graph(g_i))=graph(g_{i+1\mod n}),\forall 1\leq i\leq n.$$
Furthermore, if $g_i$ is a continuous map, then $g_i:\mathbb{T}\to \mathbb{T}^2$ is called a {\bf continuous random periodic point}.

For system $(\mathbb{T}\times \mathbb{T}^2, \varphi)$ as in \eqref{system0}, we call a system $(\mathbb{T}\times \mathbb{T},\varphi_m)$ for some $m\in \mathbb{N}$, which is \begin{equation*}
\begin{aligned}
\varphi_{m}: &\mathbb{T}\times \mathbb{T}^2 \to  \mathbb{T}\times \mathbb{T}^2\\
&(\omega', x)\mapsto (R_{\alpha/m}(\omega'), f_{m\omega'}(x)),
\end{aligned}
\end{equation*}
 the {\bf induced system} of $(\mathbb{T}\times \mathbb{T}^2, \varphi)$. 
\begin{lemma}\label{claim2.1}
For a system $(\mathbb{T}\times \mathbb{T}^2, \varphi)$ as in \eqref{system0}, and $m,n\in \mathbb{N}$, there exists a $\varphi^n$-invariant torus $\mathcal{T}$ of degree $m$, if and only if  the followings hold:
\begin{itemize}
\item[1)] there is a continuous random periodic point $\hat{g}\in C(\mathbb{T},\mathbb{T}^2)$ of period $n$ of the induced system $(\mathbb{T}\times\mathbb{T}^2,\varphi_m)$, 
\item[2)] and $\hat{g}$ is not periodic of period $\frac{1}{N}$ for any $N\in \mathbb{N}\cap [2, +\infty)$ in the case of $m\geq2$.
\end{itemize}
\end{lemma}
\begin{proof}
Assume that there is a $\varphi^n$-invariant torus of degree $m$, that is, there is a continuous map $g:\mathbb{R}\to \mathbb{T}^2$ satisfying (1)-(3) in Definition \ref{def1}. Define a continuous map $\hat{g}: \mathbb{R}\to \mathbb{T}^2$ such that $\hat{g}(\frac{\omega}{m})=g(\omega)$ for any $\omega\in \mathbb{R}$. 
It is straightforward that $\hat{g}$ is a continuous periodic map of period  $1$ since $g\in C(\mathbb{R},\mathbb{T}^2)$ is of period  $m$. 
Moreover, if $m\geq2$, by (3) in Definition \ref{def1}, $\hat{g}$ is not periodic of period $\frac1N$ for any $N\in \mathbb{N}\cap [2, +\infty)$. Due to (2) in Definition \ref{def1}, one has 
\begin{equation}\label{eqn2.2}
g(\omega+n\alpha)=f_{(n-1)\omega}\circ \cdots\circ f_{\omega}(g(\omega)).
\end{equation}
Therefore, we have
\begin{equation*}
\begin{aligned}
\varphi_m^n(\frac{\omega}{m} \mod 1, \hat{g}(\omega))&= (\frac{\omega}{m}+n\frac{\alpha}{m} \mod 1, f_{(n-1)m\frac{\omega}{m}}\circ\cdots\circ f_{m\frac{\omega}{m}}(\hat{g}(\frac{\omega}{m})))\\
&=  (\frac{\omega}{m}+n\frac{\alpha}{m} \mod 1, f_{(n-1){\omega}}\circ\cdots \circ f_{\omega}({g}({\omega})))\\
 &\overset{\eqref{eqn2.2}}{=}(\frac{\omega}{m}+n\frac{\alpha}{m} \mod 1, {g}(\omega+n\omega))\\
 &=(\frac{\omega}{m}+n\frac{\alpha}{m} \mod 1, \hat{g}(\frac{\omega}{m}+n\frac{\alpha}{m})).
\end{aligned}
\end{equation*}
Thus, $\hat{g}:\mathbb{R}\to \mathbb{T}^2$ is a continuous map satisfying 
\begin{itemize}
\item[(I)] $\hat{g}(\omega')=\hat{g}(\omega'+k)$ for all $k\in \mathbb{Z}$;
\item[(II)] $\varphi_{m}^n(\omega' \mod 1, \hat{g}(\omega'))=(\omega'+n\frac{\alpha}{m} \mod 1, \hat{g}(\omega'+n\frac{\alpha}{m}))$;
\item[(III)] $\hat{g}$ is not periodic of period $\frac1N$ for any $N\in \mathbb{N}\cap [2, +\infty)$ in the case of $m\geq2$.
\end{itemize}
Then we have a continuous random periodic point $\hat{g}\in C(\mathbb{T},\mathbb{T}^2)$ of the induced system $(\mathbb{T}\times \mathbb{T}^2,\varphi_m)$, the lift of which is $\hat{g}\in C(\mathbb{R},\mathbb{T}^2)$.

For given $m, n\in \mathbb{N}$,  assume that there is a continuous random periodic point of period $n$ of the induced system $(\mathbb{T}\times\mathbb{T}^2, \varphi_m)$, which is not periodic of period $\frac1N$ for any $N\in \mathbb{N}\cap [2, +\infty)$ in the case of $m\geq 2$.
That is, there is a $\hat{g}\in C(\mathbb{R},\mathbb{T}^2)$ satisfying (I)-(III) as above. 
Let $g(\omega)=\hat{g}\left(\frac{\omega}{m}\right)$ for any $\omega\in \mathbb{R}$, it is straightforward that $g\in C(\mathbb{R}, \mathbb{T}^2)$ satisfies (1)-(3) in Definition \ref{def1}. Then, we get a $\varphi^n$-invariant torus $\mathcal{T}$.
\end{proof}

\subsection{Existence of a ${\varphi}$-invariant torus in the case that $h$ is $C^2$}
In this section, we consider a system $(\mathbb{T}\times\mathbb{T}^2,\tilde{\varphi})$ as follows:
\begin{equation}\label{syst12}
\begin{aligned}
\tilde{\varphi}: &\mathbb{T}\times \mathbb{T}^2 \to  \mathbb{T}\times \mathbb{T}^2\\
&(\omega, x)\mapsto (R_{\alpha}(\omega), Ax+ \tilde{h}(\omega)),
\end{aligned}
\end{equation}
where
$\tilde{h}(\omega)= \begin{pmatrix}
\tilde{h}_1(\omega)\\
\tilde{h}_2(\omega)\end{pmatrix} \in C^2(\mathbb{T}, \mathbb{T}^2)$, the matrix $A\in GL(2, \mathbb{Z})$ is hyperbolic. 
\begin{lemma}\label{prop2}
For a system $( \mathbb{T}\times \mathbb{T}^2, \tilde{\varphi})$ as in \eqref{syst12}, let $m$ be the smallest positive integer such that
\begin{equation}
\label{e3}
m\cdot (I-A)^{-1}\left(\begin{matrix}deg(\tilde{h}_1)\\deg(\tilde{h}_2)\end{matrix} \right)\in\mathbb{Z}^2.
\end{equation} 
Then, there exists a $\tilde{\varphi}$-invariant torus $\mathcal{T}$ of $deg(\mathcal{T})=m$.
\end{lemma}
\begin{proof}[Proof of Lemma \ref{prop2}]
To find an invariant torus, we need to find $m, n\in \mathbb{N}$ and a periodic continuous map $g: \mathbb{R}\to \mathbb{T}^2$ satisfying (1)-(3) in Definition \ref{def1}. By Lemma \ref{claim2.1}, to get an invariant torus of $(\mathbb{T}\times \mathbb{T}^2,\tilde{\varphi})$, it is equivalent to find $m,n\in \mathbb{N}$ and a continuous mapping $\tilde{g}:\mathbb{R}\to \mathbb{T}^2$ satisfying 
\begin{itemize}
\item[(I)] $\tilde{g}(\omega')=\tilde{g}(\omega'+k)$ for all $k\in \mathbb{Z}$;
\item[(II)] $\tilde{\varphi}_{m}^n(\omega' \mod 1, \tilde{g}(\omega'))=(\omega'+n\frac{\alpha}{m} \mod 1, \tilde{g}(\omega'+n\frac{\alpha}{m}))$;
\item[(III)] $\tilde{g}$ is not periodic of period $\frac1N$ for any $N\in\mathbb{N}\cap [2, +\infty)$ in the case of $m\geq2$.
\end{itemize}
Additionally, due to  (II)  above, we have
\begin{equation}\label{eqn21}
\tilde{g}(\omega' + n\frac{\alpha}{m})= A^n \tilde{g}(\omega')+ A^{n-1}\tilde{h}(m\omega' \mod 1)+\cdots + \tilde{h}(m\omega'+(n-1)\alpha \mod 1).
\end{equation}

\medskip

{Firstly, we show when \eqref{e3} holds, we may have solutions.}
By treating $\tilde{g}=\begin{pmatrix} \tilde{g}_1\\ \tilde{g}_2\end{pmatrix}$ as a continuous map from $\mathbb{T}$ to $\mathbb{T}^2$ and comparing the degrees of both sides of  \eqref{eqn21}, we have
\begin{eqnarray*}
\begin{pmatrix} deg(\tilde{g}_1)\\ deg(\tilde{g}_2)\end{pmatrix}=A^n\begin{pmatrix} deg(\tilde{g}_1)\\ deg(\tilde{g}_2)\end{pmatrix}+ mA^{n-1}\begin{pmatrix} deg(\tilde{h}_1)\\ deg(\tilde{h}_2)\end{pmatrix} +  mA^{n-2}\begin{pmatrix} deg(\tilde{h}_1)\\ deg(\tilde{h}_2)\end{pmatrix}  +\cdots+ m \begin{pmatrix} deg(\tilde{h}_1)\\ deg(\tilde{h}_2)\end{pmatrix}.
\end{eqnarray*}
Due to the hyperbolicity of $A\in GL(2, \mathbb{Z})$, we have $\det(I-A)\neq 0$.  Therefore
\begin{equation}\label{e4}  \begin{pmatrix} deg(\tilde{g}_1)\\ deg(\tilde{g}_2)\end{pmatrix}=m \cdot (I-A)^{-1}\begin{pmatrix} deg(\tilde{h}_1)\\ deg(\tilde{h}_2)\end{pmatrix},\end{equation}
as $(I-A^n)=(I-A)(I+A+\cdots+ A^{n-1})$. Thus, for given system $(\mathbb{T}\times \mathbb{T}^2,\tilde{\varphi})$, we can only have desired $\tilde{g}\in C(\mathbb{R},\mathbb{T}^2)$ when \eqref{e3} is satisfied.

\medskip

{Secondly, we show the existence of  a continuous random fixed point of  the induced system $(\mathbb{T}\times \mathbb{T}^2, \tilde{\varphi}_{m})$.}
Let $m$ be the smallest positive integer satisfying \eqref{e3}, we get $deg(\tilde{g}_1), deg(\tilde{g}_2)\in \mathbb{Z}$ by \eqref{e4}. Let
\begin{equation}\label{e5}
\begin{cases}
\tilde{g}_1(\omega)=deg(\tilde{g}_1)\omega+\tilde{r}_1(\omega), &\text{ with $deg(\tilde{r}_1)=0$};\\
\tilde{g}_2(\omega)=deg(\tilde{g}_2)\omega+\tilde{r}_2(\omega), &\text{ with $deg(\tilde{r}_2)=0$};\\
\tilde{h}_1(\omega)=deg(\tilde{h}_1)\omega+s_1(\omega), &\text{ with $deg(s_1)=0$};\\
\tilde{h}_2(\omega)=deg(\tilde{h}_2)\omega+s_2(\omega), &\text{ with $deg(s_2)=0$}.
\end{cases}
\end{equation}
According to \eqref{eqn21} for the case that $n=1$, we have 
\begin{equation*}
\begin{aligned}
&\begin{pmatrix} deg(\tilde{g}_1)\omega+deg(\tilde{g}_1)\frac{\alpha}{m}+\tilde{r}_1(\omega+\frac{\alpha}{m})\\ deg(\tilde{g}_2)\omega+deg(\tilde{g}_2)\frac{\alpha}{m}+\tilde{r}_2(\omega+\frac{\alpha}{m})\end{pmatrix}\\
&=A\begin{pmatrix} deg(\tilde{g}_1)\omega+\tilde{r}_1(\omega)\\ deg(\tilde{g}_2)\omega+\tilde{r}_2(\omega)\end{pmatrix}+\begin{pmatrix}deg(\tilde{h}_1)m\omega+s_1(m\omega)\\ deg(\tilde{h}_2)m\omega+s_2(m\omega)\end{pmatrix}\\
&\overset{\eqref{e4}}{=}A\begin{pmatrix} \tilde{r}_1(\omega)\\ \tilde{r}_2(\omega)\end{pmatrix}+(A+I-A)\begin{pmatrix} deg(\tilde{g}_1)\omega\\ deg(\tilde{g}_2)\omega\end{pmatrix}+\begin{pmatrix} s_1(m\omega)\\ s_2(m\omega)\end{pmatrix}.
\end{aligned}
\end{equation*}
Thus, one has
\begin{equation}\label{e6}
\begin{aligned}
\begin{pmatrix}\tilde{r}_1(\omega+\frac{\alpha}{m})\\ \tilde{r}_2(\omega+\frac{\alpha}{m})\end{pmatrix}&=A\begin{pmatrix} \tilde{r}_1(\omega)\\ \tilde{r}_2(\omega)\end{pmatrix}+\begin{pmatrix} s_1(m\omega)-deg(\tilde{g}_1)\frac{\alpha}{m}\\ s_2(m\omega)-deg(\tilde{g}_2)\frac{\alpha}{m}\end{pmatrix}.
\end{aligned}
\end{equation}
Assume the Fourier series of $\tilde{r}_1, \tilde{r}_2$, $s_1$ and $s_2$ as follows
\begin{equation}\label{e7}
\begin{cases}
\tilde{r}_1(\omega)=\sum_{k\in \mathbb{Z}}a_k^1e^{2\pi i k\omega}\ \text{a.e.,}&\text{ with $a_k^1=\overline{a_{-k}^1}$};\\
\tilde{r}_2(\omega)=\sum_{k\in \mathbb{Z}}a_k^2e^{2\pi i k\omega}\ \text{a.e.,}&\text{ with $a_k^2=\overline{a_{-k}^2}$};\\
s_1(\omega)=\sum_{k\in \mathbb{Z}}b_k^1e^{2\pi i k\omega}\ \text{a.e.,}&\text{ with $b_k^1=\overline{b_{-k}^1}$};\\
s_2(\omega)=\sum_{k\in \mathbb{Z}}b_k^2e^{2\pi i k\omega}\ \text{a.e.,}&\text{ with $b_k^2=\overline{b_{-k}^2}$}.
\end{cases}
\end{equation}
According to  \eqref{e6} and the comparison of the coeffiecients of $e^{2\pi i k\omega}$ for $k\in \mathbb{Z}$, we have the following three cases. When $k=0$, one has
\begin{equation}\label{e8}
\begin{pmatrix}a_0^1\\ a_0^2\end{pmatrix}=A\begin{pmatrix} a_0^1\\ a_0^2 \end{pmatrix}+\begin{pmatrix} b_0^1-deg(\tilde{g}_1)\frac{\alpha}{m}\\ b_0^2-deg(\tilde{g}_2)\frac{\alpha}{m}\end{pmatrix}.
\end{equation}
When $k\not= 0$ and $k/m\in \mathbb{Z}$, one has
\begin{equation*}\label{e9}
e^{\frac{2\pi ik\alpha}{m}}\begin{pmatrix}a_k^1\\ a_k^2\end{pmatrix}=A\begin{pmatrix} a_k^1\\ a_k^2 \end{pmatrix}+\begin{pmatrix} b_{{k}/m}^1\\ b_{{k}/m}^2\end{pmatrix}.
\end{equation*}
Due to the hyperbolicity of $A\in GL(2, \mathbb{Z})$, one has $$ \det(e^{2\pi i k\alpha/m}I- A) \not= 0,$$
and then $e^{2\pi i k\alpha/m}I- A$ is invertible.
Therefore, we have
\begin{equation}\label{e9}
\begin{pmatrix}a_k^1\\ a_k^2\end{pmatrix}=(e^{\frac{2\pi ik\alpha}{m}}I-A)^{-1}\begin{pmatrix} b_{{k}/m}^1\\ b_{{k}/m}^2\end{pmatrix}.
\end{equation}
When $k\not= 0$ and $k/m\notin \mathbb{Z}$, by comparing the coefficients, one has
\begin{equation}\label{e10}
\begin{pmatrix}a_k^1\\ a_k^2\end{pmatrix}=\begin{pmatrix} 0\\ 0 \end{pmatrix}.
\end{equation}
Thus, according to \eqref{e8}, \eqref{e9}, and \eqref{e10}, one has
\begin{equation}\label{e11}
\begin{pmatrix}a_k^1\\ a_k^2\end{pmatrix}=
\begin{cases}
(I-A)^{-1}\begin{pmatrix} b_0^1-deg(\tilde{g}_1)\frac{\alpha}{m}\\ b_0^2-deg(\tilde{g}_2)\frac{\alpha}{m}\end{pmatrix}, &\text{when $k=0$;}\\
(e^{2\pi i k\alpha/m}I- A)^{-1}\begin{pmatrix} b_{{k}/m}^1\\ b_{{k}/m}^2\end{pmatrix}, &\text{when $k/m\in \mathbb{Z}\setminus\{0\}$;}\\
\begin{pmatrix} 0\\ 0 \end{pmatrix}, &\text{otherwise.}
\end{cases}
\end{equation}

In the next, we claim the following.
\begin{claim}\label{claim1}
For given $h_1, h_2\in C^2(\mathbb{T},\mathbb{T})$, $\tilde{r}_{1}, \tilde{r}_2$ are continuous periodic functions from $\mathbb{R}$ to $\mathbb{R}$ of period  1. 
\end{claim}
The proof of Claim \ref{claim1} can be found in the proof of Lemma 7.3 of \cite{HLL}.
By Claim \ref{claim1}, $\tilde{r}_1, \tilde{r}_2$ can be viewed as continuous functions from $\mathbb{T}$ to $\mathbb{T}$. Then for the smallest positive integer $m$ satisfying \eqref{e4}, we have a desired  $\tilde{g}\in C(\mathbb{R},\mathbb{T}^2)$. Moreover, in the case that $m\geq2$, for any $m'\in \mathbb{N}\cap (0,m)$, there is no continuous random periodic point of the induced system $(\mathbb{T}\times \mathbb{T}^2, \tilde{\varphi}_{m'})$. 
By Lemma \ref{claim1}, we have shown the existence of a $\tilde{\varphi}$-invariant torus $\mathcal{T}$ of degree $m$.
\end{proof}

\subsection{Proof of Proposition \ref{prop2.1.1}}\label{sect4.3}
To get Proposition \ref{prop2.1.1}, we consider the perturbed system $(\mathbb{T}\times \mathbb{T}^2, \tilde{\varphi}_r)$:
\begin{equation}\label{syst2}
\begin{aligned}
\tilde{\varphi}_{r}: &\mathbb{T}\times \mathbb{T}^2 \to  \mathbb{T}\times \mathbb{T}^2\\
&(\omega, x)\mapsto (R_{\alpha}(\omega), Ax+ \tilde{h}(\omega)+r(\omega)),
\end{aligned}
\end{equation}
where $\alpha\in \mathbb{R}\setminus \mathbb{Q}$, $A\in GL_{2}(\mathbb{Z})$ is hyperbolic, $\tilde{h}=\begin{pmatrix} \tilde{h}_1 \\ \tilde{h}_2\end{pmatrix}\in C^2(\mathbb{T},\mathbb{T}^2)$, and $r\in C(\mathbb{T} ,\mathbb{T}^2)$.

\begin{lemma}\label{prop4.2}
Let  $(\mathbb{T}\times \mathbb{T}^2,\tilde{\varphi}_r)$ be given as in \eqref{syst2},  if $\|r\|_{C^0}<1$,  there is an $\eta\in C(\mathbb{T},\mathbb{T}^2)$  with $deg(\eta)=\begin{pmatrix}0\\ 0\end{pmatrix}$ such that
 \begin{equation}
\label{eqn421}
\eta\circ R_{\alpha}(\omega)=A\eta(\omega)+r(\omega), \forall \omega\in \mathbb{T}.
\end{equation}
\end{lemma}
\begin{proof}
By lifting to the universal covering, finding the solution $\eta$ of \eqref{eqn421} is equivalent to solve
\begin{equation}
\label{eqn424}
\tilde{r}(\omega)=\tilde{\eta}(\omega+\alpha)-A\tilde{\eta}(\omega),
\end{equation}
where $\tilde{r}:\mathbb{R}\to \mathbb{R}^2$ is the lift of $r$, and $\tilde{\eta}:\mathbb{R}\to \mathbb{R}^2$ is the lift of $\eta: \mathbb{T}\to \mathbb{T}^2$.

Since $A$ is hyperbolic, we have a splitting $\mathbb{R}^2=E^{u}\oplus E^s$, where $E^u$/$E^{s}$ is the unstable/stable subspace, and denote by $\pi^{u}$ and $\pi^s$ the corresponding projections. For $\tau=u,s$, we have
\begin{equation*}\begin{aligned} \pi^{\tau}\tilde{r}(\omega)&=\pi^{\tau}\tilde{\eta}(\omega+\alpha)-\pi^\tau A\tilde{\eta}(\omega)=\pi^{\tau}\tilde{\eta}(\omega+\alpha)-A|_{E^\tau}\pi^\tau\tilde{\eta}(\omega),
\end{aligned}
\end{equation*}
where $A|_{E^{\tau}}$, $\tau=u,s$ represents the operator restricted on the corresponding subspace.
For any $k\in \mathbb{N}$, we have 
\begin{equation*}
(A|_{E^s})^{k-1}\pi^s\tilde{r}(\omega-k\alpha)=(A|_{E^s})^{k-1}\pi^s\tilde{\eta}(\omega-(k-1)\alpha)-(A|_{E^s})^k\pi^s\tilde{\eta}(\omega-k\alpha),
\end{equation*}
and 
\begin{equation*}
-(A|_{E^u})^{-(k+1)}\pi^u\tilde{r}(\omega+k\alpha)=(A|_{E^u})^{-k}\pi^u\tilde{\eta}(\omega+k\alpha)-(A|_{E^u})^{-k-1}\pi^u\tilde{\eta}(\omega+(k+1)\alpha).
\end{equation*}
Then, we have 
\begin{equation}\label{eqn425}
\begin{aligned}
\pi^{s}\tilde{\eta}(\omega)&=\sum_{k=1}^{+\infty}(A|_{E^s})^{k-1}\pi^s\tilde{r}(\omega-k\alpha),\\
\pi^u\tilde{\eta}(\omega)&=-\sum_{k=0}^{+\infty}(A|_{E^u})^{-k-1}\pi^u\tilde{r}(\omega+k\alpha).
\end{aligned}\end{equation}

Let $\|r\|_{C^0}<1$, one has  $deg(r)=\begin{pmatrix} deg(r_1)\\ deg(r_2)\end{pmatrix}=\begin{pmatrix}0\\0\end{pmatrix}$, where $r=\begin{pmatrix}r_1\\r_2\end{pmatrix}\in C(\mathbb{T},\mathbb{T}^2)$. 
Hence $\tilde{r}:\mathbb{R}\to \mathbb{R}^2$ is a periodic continuous mapping of period 1, 
and thus bounded. Then, we have the convergence of $\pi^u\tilde{\eta}$ and $\pi^s\tilde{\eta}$ in \eqref{eqn425}, respectively. 

According to \eqref{eqn421} and $deg(r)=\begin{pmatrix}0\\ 0\end{pmatrix}$, we have 
\begin{equation*}
\begin{pmatrix}0\\0\end{pmatrix}=deg(r)=(I-A)deg(\eta).
\end{equation*}
 Since $A\in GL(2, \mathbb{Z})$ is hyperbolic, one has $\det(I-A)\not=0$. Then, the solution $\eta\in C(\mathbb{T},\mathbb{T}^2)$  to \eqref{eqn421} should be of degree $deg(\eta)=\begin{pmatrix}0\\ 0\end{pmatrix}$.
Therefore, we need to prove that  $\tilde{\eta}:=\pi^u\tilde{\eta}+\pi^s\tilde{\eta}: \mathbb{R}\to \mathbb{R}^2$ is a periodic continuous mapping  of period $1$, which will be carried out by contradiction. 

Since $\tilde{r}$ is periodic of period  $1$, for any $i\in \mathbb{Z}, \omega\in \mathbb{R}$, one has
$$
\tilde{\eta}(\omega+\alpha)-A\tilde{\eta}(\omega)=\tilde{r}(\omega)=\tilde{r}(\omega+i)=\tilde{\eta}(\omega+i+\alpha)-A\tilde{\eta}(\omega+i).
$$
Then, for any $i\in \mathbb{Z}$, one has
\begin{equation*}
\tilde{\eta}(\omega+i+\alpha)-\tilde{\eta}(\omega+\alpha)=A\left(\tilde{\eta}(\omega+i)-\tilde{\eta}(\omega)\right),
\end{equation*}
and thus for any $i, n\in \mathbb{N}$, one has 
\begin{equation}\label{eqn426}
\begin{aligned}
\tilde{\eta}(\omega+i+n\alpha)-\tilde{\eta}(\omega+n\alpha)&=A^n(\tilde{\eta}(\omega+i)-\tilde{\eta}(\omega)),\\
\tilde{\eta}(\omega+i-n\alpha)-\tilde{\eta}(\omega-n\alpha)&=A^{-n}(\tilde{\eta}(\omega+i)-\tilde{\eta}(\omega)).
\end{aligned}
\end{equation}
Assume there is an $\omega_0\in \mathbb{R}$ such that 
$$ \tilde{\eta}(\omega_0+1)-\tilde{\eta}(\omega_0)=\begin{pmatrix}x_0\\y_0\end{pmatrix}\not=\begin{pmatrix} 0\\ 0\end{pmatrix}.$$
According to \eqref{eqn426}, one has
\begin{equation*}
\begin{aligned}
\tilde{\eta}(\omega+1+n\alpha)-\tilde{\eta}(\omega+n\alpha)&=A^n\begin{pmatrix}x_0\\ y_0\end{pmatrix},\\
\tilde{\eta}(\omega+1-n\alpha)-\tilde{\eta}(\omega-n\alpha)&=A^{-n}\begin{pmatrix} x_0\\ y_0\end{pmatrix}.
\end{aligned}
\end{equation*}
Since $A: \mathbb{R}^2\to \mathbb{R}^2$ is a given hyperbolic operator and $\tilde{\eta}$ is bounded due to \eqref{eqn425}, the equations stated above cannot hold simultaneously, which is a contradiction.

In summary, we have shown that $\tilde{\eta}\in C(\mathbb{R},\mathbb{R}^2)$ is a bounded periodic mapping of period  1. Then we get an $\eta: \mathbb{T}\to \mathbb{T}^2$, the lift of which  is $\tilde{\eta}$. Thus, we get the desired $\eta\in C( \mathbb{T}, \mathbb{T}^2)$ satisfying 
$$\eta\circ R_{\alpha}(\omega)=A\eta(\omega)+r(\omega)\text{ and }deg(\eta)=\begin{pmatrix} 0\\ 0\end{pmatrix},$$
which completes the proof of Lemma \ref{prop4.2}.
\end{proof}
Now, we show Proposition \ref{prop2.1.1} by combining Lemma \ref{prop2} and Lemma \ref{prop4.2}.
\begin{proof}[Proof of Proposition \ref{prop2.1.1}]
As $C^2(\mathbb{T},\mathbb{T}^2)$ is $C^{0}$-dense in $C(\mathbb{T},\mathbb{T}^2)$,  for any $h\in C(\mathbb{T},\mathbb{T}^2)$, there is an $\tilde{h}=\begin{pmatrix} \tilde{h}_1\\ \tilde{h}_2\end{pmatrix}\in C^2(\mathbb{T},\mathbb{T}^2)$ such that 
$\|h-\tilde{h}\|_{C^0}<1$.  Let $r=h-\tilde{h}\in C(\mathbb{T},\mathbb{T}^2)$, the system $(\mathbb{T}\times \mathbb{T}^2,\varphi)$ in \eqref{syst1} is same as the system $(\mathbb{T}\times \mathbb{T}^2,\tilde{\varphi}_r)$ in \eqref{syst2}, and $deg(r)=\begin{pmatrix} deg(h_1-\tilde{h}_1)\\ deg(h_2-\tilde{h}_2)\end{pmatrix}=\begin{pmatrix}0\\ 0\end{pmatrix}$.

According to Lemma \ref{prop2}, there is an $m\in \mathbb{N}$, $n=1$ and a continuous mapping $\hat{g}: \mathbb{R}\to \mathbb{T}^2$ satisfying (1)-(3) of Definition \ref{def1} for the system $(\mathbb{T}\times\mathbb{T}^2,\tilde{\varphi})$, where $m$ is the smallest positive integer satisfying 
$$m\cdot (I-A)^{-1}\begin{pmatrix}deg(\tilde{h}_1)\\ deg(\tilde{h}_2) \end{pmatrix}=m\cdot (I-A)^{-1}\begin{pmatrix}deg({h}_1)\\ deg({h}_2) \end{pmatrix}\in\mathbb{Z}^2.$$ 
By Lemma \ref{prop4.2}, there is an $\eta\in C(\mathbb{T},\mathbb{T}^2)$ with $deg(\eta)=\begin{pmatrix}0\\0\end{pmatrix}$. 
Then, the lifts of $r,\eta\in C(\mathbb{T},\mathbb{T}^2)$ can be seen as continuous maps from $\mathbb{R}$ to $\mathbb{T}^2$. 
That is to say, there are $\tilde{r}, \tilde{\eta}\in C(\mathbb{R},\mathbb{T}^2)$ of period $1$ satisfying
$$\tilde{\eta}({\omega}+\alpha)=A\tilde{\eta}(\omega)+\tilde{r}(\omega),\ \forall \omega\in \mathbb{R}.$$

Consider the system $(\mathbb{T}\times\mathbb{T}^2,\varphi)$ as in \eqref{syst1}, Condition (1) and (3) of Definition \ref{def1} hold for the map $\hat{g}+\tilde{\eta}\in C(\mathbb{R},\mathbb{T}^2)$. Now we prove Condition (2) of Definition \ref{def1} still holds for $\hat{g}+\tilde{\eta}$.
For any $\omega\in \mathbb{R}$, one has $$ \hat{g}(\omega+\alpha)=A\hat{g}(\omega)+\tilde{h}(\omega \mod 1).$$
Then for any $\omega\in \mathbb{R}$, one has
\begin{equation*}
\begin{aligned}
&\varphi(\omega \mod 1, \hat{g}(\omega)+\tilde{\eta}(\omega))=\tilde{\varphi}^n_r(\omega \mod 1, \hat{g}(\omega)+\tilde{\eta}(\omega))\\
&= (\omega+\alpha \mod 1, A\hat{g}(\omega)+A\tilde{\eta}(\omega)+ \tilde{h}(\omega \mod 1)+\tilde{r}(\omega \mod 1)\\
&=(\omega+\alpha \mod 1, \hat{g}(\omega+\alpha)+ \tilde{\eta}(\omega+\alpha)).
\end{aligned}
\end{equation*}
Therefore, we have shown the existence of a $\varphi$-invariant torus of degree $m$.
\end{proof}
\section{Homogeneous systems}\label{sect3}
In this section, we consider the following type of systems.
We call the system  
\begin{equation}\label{systemh}
\begin{aligned}
\varphi_0: \mathbb{T}\times \mathbb{T}^2 &\to \mathbb{T}\times \mathbb{T}^2, \\
	(\omega,x )&\mapsto (R_{\alpha}(\omega), Ax)
\end{aligned}
\end{equation}
the  {\bf homogeneous system} of  \eqref{syst1}.

Firstly, we show that if $A\in GL(2, \mathbb{Z})$ is hyperbolic, then the set of periodic points of $A: \mathbb{T}^2\to \mathbb{T}^2$ is one-to-one corresponding to the set of  invariant tori of the homogeneous system $(\mathbb{T}\times \mathbb{T}^2, \varphi_0)$.

\begin{lemma}\label{lem331}
Let $\mathcal{G}(\varphi_0)$ be the collection of all invariant tori of the homogeneous system $(\mathbb{T}\times\mathbb{T}^2, \varphi_0)$ as in Definition \ref{def1}, and $P(A):=\{x\in \mathbb{T}^2: A^nx=x, \text{ for some }n\in \mathbb{N}\}$ be the set of all periodic points of $A$. If $A\in GL(2, \mathbb{Z})$ is hyperbolic, then 
$$ \mathcal{G}(\varphi_0)= \{graph(g): g(\omega)\equiv x, x\in P(A)\}.$$
\end{lemma}
\begin{proof}
It is straightforward that $ \{graph(x): x\in P(A)\}\subseteq \mathcal{G}(\varphi_0)$. We only need to prove $\mathcal{G}(\varphi_0)\subseteq  \{graph(g): g(\omega)\equiv x, x\in P(A)\}$, which is by contradiction. 

For each $\varphi_0^{n_0}$-invariant torus $\mathcal{T}_{0}:=\{(\omega, g_{\mathcal{T}_{0}}(\omega)): \forall \omega\in \mathbb{T}\}$ of degree $m_0$, where $m_0, n_0\in \mathbb{N}$, denote the collection of all images of the multi-valued map $g_{\mathcal{T}_0}: \mathbb{T}\to \mathbb{T}^2$ by $${\rm Im}(g_{\mathcal{T}_0}):= \{g_{\mathcal{T}_0}(\omega)\in \mathbb{T}^2: \omega\in \mathbb{T}\}.$$  
Due to Definition \ref{def1}, $\mathcal{T}_0$ is a $\varphi_{0}^{n_0}$-invariant torus of degree $m_0$, where $m_0,n_0\in \mathbb{N}$, there is a continuous map $g: \mathbb{R}\to \mathbb{T}^2$, such that for any $\omega\in \mathbb{R}$, one has 
\begin{itemize}
\item[(1)] $g(\omega)=g(\omega+km_0)$ for any $k\in \mathbb{Z}$;
\item[(2)] $\varphi_0^{n_0}(\omega \mod 1, g(\omega))=(\omega+n_0\alpha \mod 1, g(\omega+n_0\alpha))$;
\item[(3)]$\{g(\omega),\dots, g(\omega+m_0-1)\}$ are pairwise distinct .
\end{itemize}
Thus, the case that  $${\rm Im}(g_{\mathcal{T}_0})=\{g_{\mathcal{T}_0}(0)\}=\{g(\omega+i): i\in \{0,\dots, m_0-1\}\}$$ has finitely many pairwise distinct values will not happen as ${\rm Im}(g_{\mathcal{T}_0})$ is path connected. Therefore, ${\rm Im}(g_{\mathcal{T}_0})$ is either path connected, which wraps around $\mathbb{T}^2$ with $m_0\geq 1$ times, or ${\rm Im}(g_{\mathcal{T}_0})$ is a singleton, which is equivalent to $\mathcal{G}(\varphi_0)\subseteq  \{graph(g): g(\omega)\equiv x, x\in P(A)\}$. 

Now, we show that the first case will not happen. For given $m_0, n_0\in \mathbb{Z}_{+}$ and a periodic continuous map $g: \mathbb{R}\to \mathbb{T}^2$ satisfying (1)-(3) as above. By Lemma \ref{claim2.1}, there is  a continuous random periodic point of period  $n_0$ of the induced homogeneous system
\begin{equation*}
\begin{aligned}
\varphi_{m_0,0}: &\mathbb{T}\times \mathbb{T}^2 \to  \mathbb{T}\times \mathbb{T}^2\\
&(\omega', x)\mapsto (R_{\alpha/m_0}(\omega'), Ax).
\end{aligned}
\end{equation*}
That is, $\tilde{g}: \mathbb{R}\to\mathbb{T}^2$ is a periodic continuous map with 
\begin{itemize}
\item[(I)] $\tilde{g}(\omega')=\tilde{g}(\omega'+k)$ for all $k\in \mathbb{Z}$;
\item[(II)] $\varphi_{m_0, 0}^{n_0}(\omega' \mod 1, \tilde{g}(\omega'))=(\omega'+n_0\frac{\alpha}{m} \mod 1, \tilde{g}(\omega'+n_0\frac{\alpha}{m}))$.
\end{itemize}

 It is clear that ${\rm Im}(\tilde{g})\setminus\{\tilde{g}(0)\}$ is path connected as $\mathbb{T}$ is path connected and $\tilde{g}$ is continuous, where $\tilde{g}$ is a continuous random periodic point of the induced homogeneous system $(\mathbb{T}\times \mathbb{T}^2,\varphi_{m_0,0})$ as  above. 
 Due to the fact that $A: \mathbb{T}^2\to \mathbb{T}^2$ is a hyperbolic linear automorphism, the global stable and unstable manifold $W^{\tau}(0), \tau=u,s$ of $0$ are orthogonal and dense in $\mathbb{T}^2$. 
 Suppose that $\sharp {\rm Im}(\tilde{g}) \geq 2$, where $\sharp$ denotes the cardinality of the corresponding subset. Then either  $({\rm Im}(\tilde{g})\setminus\{\tilde{g}(0)\}) \cap W^s(0)\not= \emptyset$ or  $({\rm Im}(\tilde{g})\setminus\{\tilde{g}(0)\}) \cap W^u(0)\not= \emptyset$.
 
The proof for these two cases are similar, we only give the proof of the first one. 
Let $\omega'\in ({\rm Im}(\tilde{g})\setminus\{\tilde{g}(0)\}) \cap W^s(0)$, as $\tilde{g}(\omega')\in W^s(0)$, we have that 
$$ \tilde{g}(\omega'+ n \frac{\alpha}{m_0})=A^n \tilde{g}(\omega')\to \tilde{g}(0) \text{ as } n\to \infty$$
Due to the fact that $\alpha\in \mathbb{R}\setminus \mathbb{Q}$, there is a subsequence $\{n_i\}_{i\in\mathbb{N}}\subseteq \mathbb{N}$ such that $$\tilde{g}(\omega'+n_i\frac{\alpha}{m_0})\to \tilde{g}(\omega') \text{ as } i\to \infty.$$
Recall the fact $\tilde{g}(\omega')\not= \tilde{g}(0)$. We have a contradiction.

In summary, for the homogeneous system $(\mathbb{T}\times \mathbb{T}^2, \varphi_0)$, one has that all invariant tori are the graphs of periodic points of the hyperbolic matrix $A\in GL(2, \mathbb{Z})$.
\end{proof}

\medskip

Secondly, we show there is a one-to-one corresponding of random periodic points between system $(\mathbb{T}\times\mathbb{T}^2,\varphi)$ and its homogeneous system $(\mathbb{T}\times\mathbb{T}^2, \varphi_0)$.
\begin{lemma}\label{lem332}
For any $n\in \mathbb{N}$, let $\mathcal{A}_0^{n}=\{\mathfrak{g}_i\}$ and   $\mathcal{A}^{n}=\{{g}_i\}$ be the set of all random periodic points of $(\mathbb{T}\times\mathbb{T}^2, \varphi_0)$ and $(\mathbb{T}\times\mathbb{T}^2, \varphi)$ of period  $n$, respectively. Then  for any $g\in \mathcal{A}^n$ and $n\in \mathbb{N}$, one has
$$ g+\mathcal{A}^n_{0}:=\{g+\mathfrak{g_i}: \mathfrak{g}_i\in \mathcal{A}^n_0\}=\mathcal{A}^n.$$
\end{lemma}
\begin{proof}
Fix $n\in \mathbb{N}$, let $g_1, g_2$ be any two random periodic points of $(\mathbb{T}\times\mathbb{T}^2, \varphi)$ of period  $n$,  then $g_1-g_2$ is a  random periodic point of period  $n$ of $(\mathbb{T}\times\mathbb{T}^2, \varphi_0)$.
This is a straightforward calculation. 
For any $\omega\in \mathbb{T}$, $i=1,2$, we have \begin{eqnarray*}
\varphi^n(\omega, g_i(\omega))&=&(R^n_{\alpha}(\omega), g_i(\omega+n\alpha))\\
&=& (R^n_{\alpha}(\omega), A^ng_i(\omega)+ \sum_{k=0}^{n-1}A^{n-1-k}h\circ R^k_{\alpha}(\omega).\end{eqnarray*}
Then, 
\begin{eqnarray*}
\varphi^n_0(\omega, (g_1-g_2)(\omega))&=&(R^n_{\alpha}(\omega), A^ng_1(\omega) -A^ng_2(\omega))\\
&=&(R^n_{\alpha}(\omega), A^ng_1(\omega)+ \sum_{k=0}^{n-1}A^{n-1-k}h\circ R^k_{\alpha}(\omega)\\
&\quad& -A^ng_2(\omega) - \sum_{k=0}^{n-1}A^{n-1-k}h\circ R^k_{\alpha}(\omega)\\
&=& (R^n_{\alpha}(\omega), g_1\circ R^n_{\alpha}(\omega)-g_2\circ R^n_{\alpha}(\omega))\\
&=& (R^n_{\alpha}(\omega), (g_1-g_2)\circ R^n_{\alpha}(\omega)).
\end{eqnarray*}
Thus, \begin{equation}\label{eqn3.4.1} \forall g\in \mathcal{A}^n, \mathcal{A}^n\subseteq g+\mathcal{A}^n_0.\end{equation}

Similarly, for any $g\in \mathcal{A}^n$, $\mathfrak{g}\in \mathcal{A}^n_0$, and $\omega\in \mathbb{T}$, one has
\begin{eqnarray*}
\varphi^n(\omega,(g+\mathfrak{g})(\omega))&=& (\omega+n\alpha, A^ng(\omega)+A^n\mathfrak{g}(\omega)+ \sum_{k=0}^{n-1}A^{n-1-k}h\circ R^k_{\alpha}(\omega) \\
&=& (\omega+n\alpha, g\circ R^n_{\alpha}(\omega)+\mathfrak{g}\circ R^n_{\alpha}(\omega)).
\end{eqnarray*}
Thus, 
$$ g+\mathfrak{g}\in \mathcal{A}^n,$$
which shows that 
\begin{equation}\label{eqn3.4.2} \forall g\in \mathcal{A}^n, g+\mathcal{A}^n_0\subseteq \mathcal{A}^n.\end{equation}
Therefore, by combining \eqref{eqn3.4.1} and \eqref{eqn3.4.2}, the lemma is proved.
\end{proof}
\begin{remark}\label{rmk}
If there is a continuous random fixed point $g\in C(\mathbb{T},\mathbb{T}^2)$ of $(\mathbb{T}\times\mathbb{T}^2, \varphi)$, then due to Lemma \ref{lem332}, we have that for any $n\in \mathbb{N}$, 
$\mathcal{A}^n\subseteq C(\mathbb{T},\mathbb{T}^2),$   
which means all random periodic points of $(\mathbb{T}\times \mathbb{T}^2,\varphi)$ are continuous.
\end{remark}
Thirdly, we show the topological entropies of $(\mathbb{T}\times \mathbb{T}^2, \varphi)$ and $(\mathbb{T}\times \mathbb{T}^2,\varphi_0)$ are equal if there is a continuous random periodic point of $(\mathbb{T}\times \mathbb{T}^2, \varphi)$.
\begin{lemma}\label{prop1}
Let $(\mathbb{T}\times \mathbb{T}^2, \varphi)$ be as in \eqref{syst1} and $(\mathbb{T}\times \mathbb{T}^2, \varphi_0)$ be its homogeneous system, if there is a continuous random periodic point $g$ of period $n\in \mathbb{N}$ of $(\mathbb{T}\times \mathbb{T}^2, \varphi)$, then $(\mathbb{T}\times \mathbb{T}^2, \varphi^n)$ and $(\mathbb{T}\times \mathbb{T}^2, \varphi^n_0)$ are topological conjugate and thus $h_{top}(\varphi)=h_{top}(\varphi_0)$.
\end{lemma}

\begin{proof}[Proof of Lemma \ref{prop1}]
By assumption, there is a continuous random periodic point  $g$ of period $n\in \mathbb{N}$ of $(\mathbb{T}\times\mathbb{T}^2, \varphi)$, i.e., $g\in C(\mathbb{T},\mathbb{T}^2)$ and 
$$\varphi^n(\omega, g(\omega))= (R^n_{\alpha}(\omega), g\circ R^n_{\alpha}(\omega))=(R^n_{\alpha}(\omega), A^ng(\omega)+ \sum_{k=0}^{n-1}A^{n-1-k}h\circ R^k_{\alpha}(\omega)).$$ 
Denote 
\begin{eqnarray*}
T_{g}:  &\mathbb{T}\times \mathbb{T}^2\to \mathbb{T}\times \mathbb{T}^2 \\
&(\omega, x) \mapsto (\omega, x+g(\omega)),
\end{eqnarray*}
and 
\begin{eqnarray*}
\tilde{T}_{g}:  &\mathbb{T}\times \mathbb{T}^2\to \mathbb{T}\times \mathbb{T}^2 \\
&(\omega, x) \mapsto (\omega, x-g(\omega)).
\end{eqnarray*}
It is clear that $T_g\circ \tilde{T}_g=\tilde{T}_g\circ T_g=id$, and $T_g$, $\tilde{T}_g$ are continuous as $g$ is continuous. Thus, $T_{g}$ is a homeomorphism. 

For any $(\omega, x)\in \mathbb{T}\times \mathbb{T}^2$, we have
\begin{eqnarray*}
\varphi^n \circ T_g(\omega,x)&=&\varphi^n(\omega, x+g(\omega))\\
&=&(R^n_{\alpha}(\omega),  A^n(x+g(\omega))+ \sum_{k=0}^{n-1}A^{n-1-k}h\circ R^k_{\alpha}(\omega))\\
&=&(R^n_{\alpha}(\omega),  A^n(x)+A^ng(\omega)+\sum_{k=0}^{n-1}A^{n-1-k}h\circ R^k_{\alpha}(\omega)),
\end{eqnarray*}
and 
\begin{eqnarray*}
T_g\circ \varphi^n_0(\omega,x)=T_g(R^n_{\alpha}(\omega), A^nx)
=(R^n_{\alpha}(\omega),  A^nx+g\circ R^n_{\alpha}(\omega)).
\end{eqnarray*}
Since, $g\circ R^n_{\alpha}(\omega)=A^ng(\omega)+\sum_{k=0}^{n-1}A^{n-1-k}h\circ R^k_{\alpha}(\omega)$, 
we have $\varphi\circ T_g= T_g\circ \varphi_0$, i.e., 
\[\begin{array}{ccc}
\mb T\times \mb T^2 &
\stackrel{\varphi_0}{\longrightarrow} &
\mathbb T\times \mathbb T^2 \\
\Big\downarrow \rlap{$T_g$} & &
\Big\downarrow\vcenter{%
\rlap{$T_g$}}\\
\mathbb T\times \mathbb T^2 & \stackrel{\varphi}{\longrightarrow} &
\mathbb T\times \mathbb T^2
\end{array}.\]
Therefore, $(\mathbb{T}\times \mathbb{T}^2, \varphi^n)$ and $(\mathbb{T}\times \mathbb{T}^2, \varphi^n_0)$ are topological conjugate. Thus $$h_{top}(\varphi^n_0)=h_{top}(\varphi^n)$$ as $h_{top}(\varphi_0)<+\infty$ and the topological entropy is a topological invariant. Then, one has 
$$h_{top}(\varphi_0)=h_{top}(\varphi),$$
which completes the proof of Lemma \ref{prop1}.
\end{proof}

\begin{remark}
There is no need to prove the topological conjugacy between  $(\mathbb{T}\times \mathbb{T}^2, \varphi^n)$ and $(\mathbb{T}\times \mathbb{T}^2,\varphi^n_0)$ for some $n\in \mathbb{N}$ to get  $h_{top}(\varphi)=h_{top}(\varphi_{0})$. 
However, with the help of continuous random periodic points, we can get the topological conjugacy between these two systems straightforwardly, which may not hold without them.
\end{remark}

\section{Proof of Theorem \ref{thm1}}\label{sect4}
The proof is done by combining Proposition \ref{prop2.1.1}, Lemma \ref{lem331}, Lemma \ref{lem332} and Lemma \ref{prop1}.

Due to Proposition \ref{prop2.1.1}, there is an $m\in \mathbb{N}$ and a $\varphi$-invariant torus of degree $m$, where $m$ is the smallest positive integer satisfying $m\cdot (I-A)^{-1}\begin{pmatrix} deg(h_1)\\ deg(h_2)\end{pmatrix}\in \mathbb{Z}^2$. Moreover, due to Lemma \ref{claim2.1}, there is a continuous random fixed point $g\in C(\mathbb{T},\mathbb{T}^2)$ of the induced system 
\begin{equation*}
\begin{aligned}
\varphi_{m}: &\mathbb{T}\times \mathbb{T}^2 \to  \mathbb{T}\times \mathbb{T}^2\\
&(\omega', x)\mapsto (R_{\alpha/m}(\omega'), Ax+ h(m\omega' \mod 1)).
\end{aligned}
\end{equation*}
Here, we introduce a finite-to-one map 
\begin{eqnarray*}
\mathcal{K}_{m}:  &\mathbb{T}\times \mathbb{T}^2\to \mathbb{T}\times \mathbb{T}^2 \\
&(\omega, x) \mapsto (m\omega \mod 1, x).
\end{eqnarray*}
Then, we have the following commuting diagram, i.e., 
\[\begin{array}{ccc}
\mb T\times \mb T^2 &
\stackrel{\varphi_m}{\longrightarrow} &
\mathbb T\times \mathbb T^2 \\
\Big\downarrow \rlap{$\mathcal{K}_{m}$} & &
\Big\downarrow\vcenter{%
\rlap{$\mathcal{K}_{m}$}}\\
\mathbb T\times \mathbb T^2 & \stackrel{\varphi}{\longrightarrow} &
\mathbb T\times \mathbb T^2
\end{array}\ .\]
Thus, we have \begin{equation}
\label{eqn341}
h_{top}(\varphi)=h_{top}(\varphi_m).
\end{equation}
Since $g\in C(\mathbb{T},\mathbb{T}^2)$ is a continuous fixed point of the induced system $(\mathbb{T}\times\mathbb{T}^2, \varphi_m)$, by \eqref{eqn341} and Lemma \ref{prop1}, we have 
\begin{equation}\label{eqn342}
h_{top}(\varphi)=h_{top}(\varphi_m)=h_{top}(\varphi_{m,0}),
\end{equation} 
where $(\mathbb{T}\times \mathbb{T}^2,\varphi_{m,0})$ is the homogeneous system of the induced system $(\mathbb{T}\times \mathbb{T}^2,\varphi_{m})$.
Moreover, let $$T_{A}: \mathbb{T}^2\to \mathbb{T}^2,\ x\mapsto Ax,$$
where $A\in GL(2, \mathbb{Z})$ is a hyperbolic matrix.
By combining with the fact that $h_{top}(R_{\alpha/m})=0$, we have
\begin{equation}\label{eqn343}
h_{top}(\varphi_{m,0})= h_{top}(R_{\alpha/m})+ h_{top}(T_{A})=\lim_{n\to \infty}\frac1n \log \sharp P(A; n),
\end{equation}
where $P(A;n)$ is the collection of all periodic points of $A$ with periods less than or equal to $n$.

Now we consider the relationship between the cardinalities of the periodic points of the hyperbolic matrix $A\in GL(2, \mathbb{Z})$ and the invariant tori of $(\mathbb{T}\times \mathbb{T}^2,\varphi)$. Again due to Lemma \ref{prop2},  there is a $\varphi$-invariant torus of degree $m$, which can induce a continuous random fixed point $g\in C(\mathbb{T},\mathbb{T}^2)$ of the induced system $(\mathbb{T}\times\mathbb{T}^2, \varphi_m)$. Then, due to Lemma \ref{lem331}, Lemma \ref{lem332} and Remark \ref{rmk}, we have 
\begin{equation}\label{eqn344}
\sharp P(A; n)=\sharp \mathcal{G}(\varphi_{m,0}; n) =\sharp \cup_{k\leq n}\mathcal{A}_0^k=\sharp \cup_{k\leq n}\mathcal{A}^k=\sharp \mathcal{G}(\varphi; n, m),
\end{equation}
where $\mathcal{G}(\varphi_{m,0}; n)$ is the collection of all $
\varphi_{m,0}^k$-invariant tori with $k\leq n$, $\mathcal{A}_0^k$, $\mathcal{A}^k$ are the sets of all random periodic points with period $k$ of $(\mathbb{T}\times \mathbb{T}^2,\varphi_{m,0})$ and $(\mathbb{T}\times \mathbb{T}^2, \varphi_{m})$ as in Lemma \ref{lem332}, and $\mathcal{G}(\varphi; n, m)$ is the collection of $\varphi^k(k\leq n)$-invariant tori of degree $m$.

At the last, we show the uniqueness of $m\in \mathbb{N}$. Since $m$ is the smallest positive integer satisfying \eqref{e3}, due to the discussion in Lemma \ref{prop2}, there are continuous random periodic points of $(\mathbb{T}\times\mathbb{T}^2, \varphi_{lm})$, for all $l\in\mathbb{N}$. It is straightforward that all random periodic points of $(\mathbb{T}\times\mathbb{T}^2,\varphi_m)$ are random periodic points of $(\mathbb{T}\times\mathbb{T}^2,\varphi_{lm})$, for any $l\in \mathbb{N}$, that is 
\begin{equation}\label{eqn345}
\mathcal{A}^n_{m}\subseteq \mathcal{A}^n_{lm},\forall n,l\in \mathbb{N},
\end{equation}
where $\mathcal{A}_m^n$, $\mathcal{A}_{lm}^n$ are the sets of all random periodic points with period $n$ of $(\mathbb{T}\times \mathbb{T}^2,\varphi_{m})$ and $(\mathbb{T}\times \mathbb{T}^2, \varphi_{lm})$. 
Due to Lemma \ref{lem331}, Lemma \ref{lem332}, and the fact the homogeneous systems of $(\mathbb{T}\times\mathbb{T}^2,\varphi_{lm}),\forall l\in \mathbb{N}$ are same, we have 
\begin{equation}\label{eqn346}
\sharp P(A; n)=\sharp \mathcal{G}(\varphi_{lm,0}; n) =\sharp \cup_{k\leq n}\mathcal{A}_{lm,0}^k=\sharp \cup_{k\leq n}\mathcal{A}_{lm}^k,
\end{equation}
where $\mathcal{A}_{lm, 0}^k$ and $\mathcal{A}_{lm}^k$ are the sets of all random periodic points of $(\mathbb{T}\times \mathbb{T}^2,\varphi_{lm,0})$ and $(\mathbb{T}\times \mathbb{T}^2, \varphi_{lm})$, respectively. By combining \eqref{eqn345} and \eqref{eqn346}, one has 
\begin{equation*}\label{eqn347}
\mathcal{A}^n_{m}= \mathcal{A}^n_{lm},\forall n,l\in \mathbb{N}.
\end{equation*}
That is to say all continuous random periodic points of $(\mathbb{T}\times \mathbb{T}^2, \varphi_{lm})$ for $l\in\mathbb{N}$ are continuous random periodic points of $(\mathbb{T}\times \mathbb{T}^2, \varphi_{m})$. 
Due to (3) of Definition \ref{def1}, we only have $\varphi^n$-invariant tori of degree $m$.

By combining the above argument with \eqref{eqn342}, \eqref{eqn343} and \eqref{eqn344}, we have proven Theorem \ref{thm1}. $\hfill\square$





\medskip

\noindent{\bf Acknowledgement:}
The authors would like to thank  Wen Huang, Yi Shi, and Hui Xu for useful discussions.
Z. Lian is partially supported by NNSF of China (12090012). 
X. Ma is partially supported by NNSF of China (12471188, 12090012), USTC Research Funds of the Double First-Class Initiative, and the Fundamental Research Funds for the Central Universities. H. Zhao is partially supported by NNSF of China (123B2006).







\bibliographystyle{plain}

\end{document}